\documentclass[12pt]{amsart}
\usepackage{amsfonts,bbm,esint,mathrsfs,amssymb,amsmath,amsfonts,amsthm}
\usepackage[colorlinks, citecolor=blue, pagebackref, hypertexnames=false]{hyperref}
\usepackage{mathtools}

\usepackage{graphicx,scalerel}
\usepackage{subcaption}

\usepackage{caption}  

\usepackage[most]{tcolorbox}

\numberwithin{equation}{section}

\newtheorem{thm}{Theorem}[section]
\newtheorem{lem}[thm]{Lemma}

\theoremstyle{definition}

\def \12 {{\frac{1}{2}}}

\def \x8n {\int_{8 \leq |x_1| \leq N_1}}
\def \y8n {\int_{8 \leq |x_2| \leq N_2}}
\def \z8n {\int_{8 \leq |x_3| \leq N_4}}
\def \xe8 {\int_{\epsilon_1 \leq |x_1| \leq 8}}
\def \ye8 {\int_{\epsilon_2 \leq |x_2| \leq 8}}
\def \ze8 {\int_{\epsilon_4 \leq |x_3| \leq 8}}

\allowdisplaybreaks

\newcommand{\charac}{\raisebox{2pt}{$\chi$}}

\begin{document}

\baselineskip 17.2pt \hfuzz=6pt

\title[Zygmund type and flag type maximal functions]
{Zygmund type and flag type maximal functions, \\ and sparse operators}

\bigskip

\author[Flores, G., Li, J.\ and Ward, L. ]{Guillermo J. Flores, Ji Li and Lesley A. Ward}

\address{Guillermo J. Flores, FI Univ.\ Cat\'olica de C\'ordoba, Av.\ Armada Argentina 3555 CP: X5016DHK
\& CIEM-FAMAF Univ.\ Nacional de C\'ordoba, Av.\ Medina Allende s/n Ciudad Univ. CP:X5000HUA,
C\'ordoba, Argentina.}
\email{gflores@famaf.unc.edu.ar}

\address{Ji Li, Department of Mathematics\\
Macquarie University\\ 
NSW, 2109, Australia.}
\email{ji.li@mq.edu.au}

\address{Lesley A. Ward, School of Information Technology and Mathematical Sciences, University of South Australia, 
Mawson Lakes SA 5095, Australia.  }
\email{lesley.ward@unisa.edu.au}

\thanks{The authors are supported by ARC DP 160100153}


\begin{abstract}
We prove that the maximal functions associated with a Zygmund dilation {dyadic} structure in
three-dimensional Euclidean space, and with the flag {dyadic} structure in two-dimensional
Euclidean space, cannot be bounded by multiparameter sparse operators associated
with the corresponding dyadic grid. {We also obtain supplementary results about the absence of 
sparse domination for the strong dyadic maximal function.} 
\end{abstract}

\maketitle

%\tableofcontents

%%%%%%%%%%%%%%%%%%%%%%%%%%%%%%%%%%%%%%%%%%%%%%%%%%%%%%%%%%%%%%%%%%%%%%%%%%%%%%%%%%%%%%%%%
%%%%%%%%%%%%%%%%%%%%%%%%%%%%%%%%%%%%%%%%%%%%%%%%%%%%%%%%%%%%%%%%%%%%%%%%%%%%%%%%%%%%%%%%%
\section{Introduction and statement of main results}

In recent years, it has been evidenced that \emph{Sparse Operators} play an important 
role in the weighted bounds for many singular integrals, see for example \cite{La1,Le,CCDO, CR}. Such techniques
have led to advances in sharp estimates within the Calder\'on-Zygmund theory.
The fundamental example is the sparse domination of the one-parameter dyadic maximal function 
\begin{align*}
M_d f (x) \vcentcolon = \sup_{Q \in \mathscr{D}^n \, : \, x \in Q} {1\over |Q|} \int_Q |f(x_1)| \, dx_1
\end{align*}
where the supremum is taken over all dyadic cubes in $\mathbb R^n$ containing $x$, that is,
\begin{align*}
M_d f (x) \leq C \sum_{Q \in \mathscr{S}} \left( {1\over |Q|} \int_Q f(x_1) \, dx_1 \right) \charac_Q(x),
\end{align*}
where $\mathscr{S}$ is a sparse collection of dyadic cubes.

Nevertheless, a remarkable recent result (Theorem A in \cite{BCOR}) shows that there is no sparse 
domination in the tensor product setting $\mathbb R^n\times \mathbb R^m$ for the strong dyadic maximal function
\begin{align*}
M_{sd} f (x,y) \vcentcolon = \sup_{R \in \mathscr{D}^{n} \times \mathscr{D}^{m} \, : \, (x,y) \in R} {1\over |R|} \int_R |f(x_1,y_1)| \, dx_1dy_1,
\end{align*}
where the supremum is taken over all dyadic rectangles with sides parallel to the axes containing $(x,y)$. This result suggests that
the sparse domination techniques in the one-parameter setting cannot be expected to work 
with the same approach in the multiparameter setting.

The classical Calder\'on--Zygmund
singular integrals are related to the one-parameter dilation
structure on $\Bbb R^n$, defined by $\delta_o \circ
(x_1,x_2,\ldots, x_n) \vcentcolon = (\delta x_1,\ldots,\delta x_n)$, with
$x\in\mathbb{R}^n$ and $\delta > 0$. Meanwhile the product
dilation structure is defined by $\delta_p \circ (x_1,x_2,\ldots, x_n) \vcentcolon =
(\delta_1 x_1,\ldots,\delta_n x_n)$, $\delta_i > 0$, $i =
1,\ldots,n$. The key difference is that $\delta_o$ maps cubes
to cubes, while $\delta_p$ maps cubes to rectangular prisms
whose side-lengths are independent. Multiparameter dilations
lie between these two extremes: the side-lengths need not be
equal nor be completely independent of each other, but may be mutually dependent. 

With these dilation structures in mind, it is natural to wonder whether it is possible to obtain certain 
sparse domination for multiparameter maximal functions which lie in between
the two extreme cases $M_d$ and $M_{sd}$.

One of the most natural and interesting examples of a group of dilations in $\Bbb R^3$ that lies in between 
the one-parameter and the full product setting is the so-called \emph{Zygmund dilation} defined by
$\rho_{s,t}(x_1,x_2,x_3) = (sx_1, tx_2, stx_3)$ for $s, t > 0$ (see for example \cite{RS,FP}). The 
maximal function corresponding to this Zygmund dilation is  
\begin{align}\label{zygmund maximal}
{\mathcal M}_{\frak z} f (x_1,x_2,x_3) \vcentcolon = \sup_{R \, : \, (x_1,x_2,x_3) \in R} \frac{1}{|R|}\int_{R}|f(u_1,u_2,u_3)| \, du_1du_2du_3,
\end{align}
where the supremum above is taken over all rectangles in $\Bbb R^3$ with edges parallel
to the axes and side-lengths of the form $s, t$, and $st$ (see \cite{Co}). See also \cite{So} for a discussion
of the Zygmund conjecture about the differentiation properties of $k$-parameter bases of rectangular 
prisms in $\mathbb{R}^n$, and \cite{DLOWY2} for the Zygmund type singular integrals and their commutators. The survey paper of R. Fefferman \cite{Fe1} has more information 
about research directions in this setting.

Another very important example in the multiparameter setting is the implicit flag structure.
To be precise, in \cite{MRS,MRS2}, M\"{u}ller, Ricci and Stein studied Marcinkiewicz 
multipliers on the Heisenberg group $\mathbb H^n$  associated with the sub-Laplacian on 
$\mathbb H^n$ and the central invariant vector field, and obtained the $L^p$-boundedness for $1<p<\infty$. 
This is surprising since these multipliers are invariant under a two-parameter
group of dilations on $\mathbb{C}^{n}\times \mathbb{R}$, while there is
\emph{no} two-parameter group of \emph{automorphic} dilations on $\mathbb{H}
^{n}$. Moreover, they showed that Marcinkiewicz multipliers can be characterized by
a convolution operator of the form $f\ast K$ where $K$ is a \emph{flag} convolution kernel, 
which satisfies size and  smoothness conditions lying in between the one-parameter and 
product singular integrals. The complete flag Hardy space theory and the boundedness of the iterated commutator was obtained only recently in \cite{HLLW} and \cite{DLOWY}, respectively.
The fundamental tool in this setting is the  flag maximal function. We state the definition in $\mathbb R\times \mathbb R$
for the sake of simplicity: 
\begin{align}\label{flag maximal}
{\mathcal M}_{\text {flag}} f (x_1,x_2) \vcentcolon = \sup_{R \, : \, (x_1,x_2) \in R}\frac{1}{|R|}\int_{R}|f(u_1,u_2)|du_1du_2,
\end{align}
where the supremum above is taken over all rectangles in $\Bbb R^2$ with edges parallel
to the axes and side-lengths of the form $s$ and $t$ satisfying
$s\leq t$.

%Note that the dyadic versions of the Zygmund maximal function and flag maximal function can be defined easily by replacing the corresponding
%rectangles in \eqref{zygmund maximal} and \eqref{flag maximal} by dyadic rectangles, respectively. We denote them by ${\mathcal M}_{\frak z,d}  $ 
%and ${\mathcal M}_{flag,d}$, respectively.
{The dyadic versions of the Zygmund maximal function and flag maximal function can be defined easily by restricting to 
dyadic axis-parallel rectangles in \eqref{zygmund maximal} and \eqref{flag maximal}.
%, keeping the corresponding structure
%conditions over these rectangles. 
We denote them by ${\mathcal M}_{\frak z,d}$ 
and ${\mathcal M}_{{\rm flag},d}$, respectively.}

In this article we show that the maximal functions $ {\mathcal M}_{\frak z,d} $ and $ {\mathcal M}_{{\rm flag},d} $ cannot be bounded 
by multiparameter sparse operators associated with the corresponding dyadic grid. We state these results as 
Theorems \ref{Thm Z} and \ref{Thm F}, respectively.
\begin{thm}\label{Thm Z}
Take $r, s \geq 1$ such that $1/r + 1/s > 1$. Then for every $C>0$ and $\eta \in (0,1)$
there exist integrable functions $f$ and $g$, compactly supported and bounded, such that
$$\big|\langle \mathcal M_{\mathfrak z,d} f,g \rangle\big| \geq
C \sum_{R\in \mathscr{S}_{\mathfrak z}} \langle |f| \rangle_{R,r} \langle |g| \rangle_{R,s}  |R|,$$
for all $\eta$-sparse collections $\mathscr{S}_{\mathfrak z}$ of Zygmund dyadic {edge-parallel} rectangles. 
\end{thm}
Here we are denoting the $L^r$-average of a function $f$ over a rectangle $R$ by 
\begin{align*}
\langle |f| \rangle_{R,r} \vcentcolon = \left( \frac{1}{|R|} \int_R |f|^r \right)^{1/r} , \quad \hbox{for } r \geq 1 .
\end{align*}
\begin{thm}\label{Thm F}
Take $r, s \geq 1$ such that $1/r + 1/s > 1$. Then for every $C>0$ and $\eta \in (0,1)$
 there exist integrable functions $f$ and $g$, compactly supported and bounded, such that
 \begin{align*}
 \big|\langle \mathcal M_{{\rm flag},d} f,g \rangle\big| \geq 
C \sum_{R\in \mathscr{S}_{\rm flag}} \langle |f| \rangle_{R,r} \langle |g| \rangle_{R,s} |R|,
 \end{align*}
for all $\eta$-sparse collections $\mathscr{S}_{\rm flag}$ of flag dyadic {edge-parallel} rectangles. 
\end{thm}

%Also, we consider the strong maximal function given by 
%\begin{align*}
%M_s f (x) \vcentcolon = \sup_{R \, : \, x \in R} \frac{1}{|R|} \int_R |f(x)| dx
%\end{align*}
%where the supremum is taken over all rectangles in $\mathbb{R}^n$ with sides parallel to the axes. Then,
Also, we show that the strong dyadic maximal function $M_{sd}$ does not admit $(r,s)$-sparse domination for certain $r$ and $s$, in the following result.
\begin{thm}\label{Thm ZS}
Take $r, s \geq 1$ such that $1/r + 1/s > 1$. Then for every $C>0$ and $\eta \in (0,1)$
there exist integrable functions $f$ and $g$, compactly supported and bounded, such that
$$\big|\langle M_{sd} f,g \rangle\big| \geq
C \sum_{R\in \mathscr{S}} \langle |f| \rangle_{R,r} \langle |g| \rangle_{R,s}  |R|,$$
for all $\eta$-sparse collections $\mathscr{S}$ of dyadic {edge-parallel} rectangles. 
\end{thm}
This provides a supplementary explanation to the main result of \cite{BCOR}.

%This theorem strengthens the main result of \cite{BCOR}, Theorem A. Indeed, Theorem A is the special case of 
%Theorem \ref{Thm ZS} where $r=1$ and $s=1$.

%We remark that there are no direct implications among the previously stated theorems. For instance, in biparameter
%setting $M_{sd}$ is greater than $\mathcal M_{flag,d}$, 
%but the sparse family of flag dyadic rectangles does not have a direct comparison with the sparse family of dyadic rectangles.
We remark that there are no direct implications among the previously stated theorems. For instance, in two-parameter
setting $M_{sd}$ is greater than $\mathcal M_{{\rm flag},d}$. {But the sums over the sparse collections of flag dyadic rectangles
involved in Theorem \ref{Thm F} do not have a direct comparison with the sums over the sparse collections of dyadic rectangles
involved in Theorem \ref{Thm ZS}. A similar situation occurs for $M_{sd}$ and $\mathcal M_{\mathfrak z,d}$.}

The paper is organised as follows. In Section \ref{Zygmund} we give notation 
and some key results, which we  then use to prove Theorem \ref{Thm Z}. In Section \ref{Flag} we prove Theorem \ref{Thm F},
and in Section \ref{strong} we prove Theorem \ref{Thm ZS}.

\section{Zygmund dilation dyadic structures}
\label{Zygmund}
\subsection{Notation and proof of Theorem \ref{Thm Z}}
%In this section we shall give a slightly more general result than Theorem \ref{Thm Z}, stated for 
%the dyadic version of the Zygmund maximal function:
%\begin{align}\label{zygmund dyadic maximal}
%{\mathcal M}_{\frak z,d} f (x_1,x_2,x_3) \vcentcolon = 
%\sup_{R \in \mathscr{D}_{\mathfrak{z}} \, : \, (x_1,x_2,x_3) \in R}\frac{1}{|R|}\int_{R}|f(u_1,u_2,u_3)| \, du_1du_2du_3,
%\end{align}
%where the supremum above is taken over all dyadic rectangles $R=I\times J\times S$ in $\mathscr{D}^3$
%such that $|S|=|I|\cdot |J|$, denoted by $\mathscr{D}_{\mathfrak{z}} $.
As usual, the collection $\mathscr{D}$ of dyadic intervals in $\mathbb{R}$ is defined by
\begin{align*} 
\mathscr{D} = \big\{ R \subset \mathbb{R} \, : \, R = \big[ k 2^{-j} , (k+1) 2^{-j} \big) \hbox{ for } k, \, j \in \mathbb{Z} \big\} .
\end{align*}
{Then, we shall denote by $\mathscr{D}_{\mathfrak{z}} $ the collection of
all \emph{Zygmund dyadic rectangles} $R=I\times J\times S$ in $\mathscr{D}^3$, that is, those 
$R\in \mathscr{D}^3$ such that $|S|=|I|\cdot |J|$.}

We shall evaluate
$ {\mathcal M}_{\frak z,d} $ on finite sums of special point masses. 

In general, given a locally integrable function $f$ we can define the associated measure $ \mu_f $ by $\mu_f (E):= \int_E |f|$
for every measurable set $E$. Then
\begin{align*}
{\mathcal M}_{\frak z,d} f (x) \vcentcolon = 
\sup_{R \in \mathscr{D}_{\mathfrak{z}} \, : \, x \in R} \frac{1}{|R|} \int_{R} |f| = 
\sup_{R \in \mathscr{D}_{\mathfrak{z}}  \, : \, x \in R} \frac{\mu_f (R)}{|R|} = \vcentcolon {\mathcal M}_{\frak z,d} \, \mu_f \, (x) .
\end{align*} 
Now, if $ \mathcal{F} $ is a finite set in $\Bbb R^n$ we define the finite sum $\mu$ of point masses associated with $ \mathcal{F} $ by
\begin{align*}
\mu \, \vcentcolon = \, \frac{1}{\sharp \mathcal{F}} \sum_{p \in \mathcal{F}} \delta_p ,
\end{align*}
where $ \sharp \mathcal{F} $ denotes the number of points in $ \mathcal{F} $ 
and $\delta_p$ denotes the single point mass concentrated at $p$. Then we naturally make sense of
\begin{align*}
& {\mathcal M}_{\frak z,d} \, \mu \, (x) \, \vcentcolon = \, \frac{1}{\sharp \mathcal{F}} \, 
\sup_{R \in \mathscr{D}_{\mathfrak{z}} \, : \, x \in R} \frac{1}{|R|} \sum_{p \in \mathcal{F}} \delta_p (R) ,
\qquad \langle f, \mu \rangle \vcentcolon = \int f \, d \mu = \frac{1}{\sharp \mathcal{F}} \sum_{p \in \mathcal{F}} f(p), 
\end{align*}
and
\begin{align*}
\langle \mu \rangle_{R,r} \vcentcolon = \, \frac{1}{\sharp \mathcal{F}} \bigg( \frac{1}{|R|} \sum_{p \in \mathcal{F}} \delta_p (R) \bigg)^{1/r} 
\end{align*}
for every rectangle $R$ and $1 \leq r < \infty$.

Next, given $\eta$ with $0 < \eta < 1$, we shall say that a collection $ \mathscr{S} $ of sets of finite measure 
(usually, rectangles or even dyadic {edge-parallel} rectangles) is called $\eta$-\emph{sparse}, if for each $R \in \mathscr{S}$
there is a subset $E_R \subset R$ such that $ |E_R| \geq \eta |R| $, and the collection $\{ E_R \}$ is pairwise disjoint.
Then, for $r, s \geq 1$, we say that an operator $T$ \emph{admits an $(r,s)$ $\eta$-sparse domination} if 
\begin{align*}
\big|\langle T f,g \rangle\big| \leq
C \sum_{R\in \mathcal S } \langle |f| \rangle_{R,r} \langle |g| \rangle_{R,s} |R|,
\end{align*}
for every pair of functions $f$ and $g$ sufficiently nice in the given context.

{We are now in position to state the following key result from which we will deduce Theorem \ref{Thm Z}
as a corollary.} We are denoting by $C$ and $c$ positive constants,  not necessarily 
the same at each occurrence.
\begin{thm}\label{ZD}
Let $r, s \geq 1$ such that $1/r + 1/s > 1$. Then for every natural number $k$ and for every $\eta \in (0,1)$
there exist finite sums $\mu_k$ and $\nu_k$ of point masses in $\mathbb{R}^3$ such that
\begin{itemize}
\item[(a)] $ \displaystyle \langle {\mathcal M}_{\frak z,d} \, \mu_k \, , \, \nu_k \rangle \, \geq \, c \, 2^k $, and
\item[(b)] $ \displaystyle \sum_{R\in \mathscr{S}_{\frak z}} \langle \mu_k \rangle_{R,r} \langle \nu_k \rangle_{R,s}  |R| 
\leq \frac{C}{\eta} \, k^{1/s} \left( 1 + \frac{1}{k} \right) $ for all $\eta$-sparse collections $\mathscr{S}_{\mathfrak z}$ 
of Zygmund dyadic rectangles. 
\end{itemize}
\end{thm}
We can deduce Theorem \ref{Thm Z} as follows.
If we assume that $ {\mathcal M}_{\frak z,d} $ admits an $(r,s)$ $\eta$-sparse domination with 
$1/r + 1/s > 1$, then for each $k \in \mathbb{N}$ we have
\begin{align*}
\langle {\mathcal M}_{\frak z,d} \, \mu_k \, , \, \nu_k \rangle \, \leq \, C 
\sum_{R\in \mathscr{S}_{\frak z}} \langle \mu_k \rangle_{R,r} \langle \nu_k \rangle_{R,s} |R| ,
\end{align*}
for $\mu_k$ and $\nu_k$ as in Theorem \ref{ZD}. But the latter forces $\eta = 0$ by (a) and (b) of the previous theorem, 
which leads to a contradiction. {Therefore, by using a limiting and approximation argument, 
we obtain a proof of Theorem \ref{Thm Z}.}
\subsection{Construction of special finite sums of point masses}
Now we shall give the explicit formulas of $\mu_k$ and $\nu_k$ in Theorem \ref{ZD}. 
For brevity we drop the subscript $k$ from $\mu_k$ and $\nu_k$. Also,
for our proof below we need one more auxiliary result.

The authors in \cite{BCOR} introduced a 
\emph{dyadic distance} function given by
\begin{align*}
d_{\mathscr{D}} (p,q) \vcentcolon = \inf \big\{ |R|^{1/2} \, : \, R \in \mathscr{D}^n \hbox{ and } p, \, q \in R \big\},
\end{align*}
for every pair of points $p$ and $q$ in the cube $[0,1)^n$. The function $d_{\mathscr{D}}$ turns out to be intuitive 
in terms of the geometry in the dyadic {size-parallel} rectangles setting. We note that this function does not satisfy 
the conditions of a true distance, as remarked in \cite{BCOR}, but nevertheless we shall refer to $d_{\mathscr{D}}$ as 
\emph{the dyadic distance} between two points in $[0,1)^n$. 

Next, associated with the \emph{Zygmund dilation structure}, 
let $ d_{\mathscr{D}_{\mathfrak{z}}} $ be the \emph{Zygmund dyadic distance} given by
\begin{align*}
d_{\mathscr{D}_{\mathfrak{z}}} (p,q) \vcentcolon = \inf \big\{ |R|^{1/2} \, : \, R \in \mathscr{D}_{\mathfrak{z}} \hbox{ and } p, \, q \in R \big\} ,
\end{align*}
for every pair of points $p$ and $q$ in the cube $[0,1)^3$. 
\begin{lem}\label{main lemma zygmund}
For every natural number $k$, set $m = k 2^{6k}$. Then there exist two sets of points $\mathcal P_{\mathfrak z}$ 
and $\mathcal Z_{\mathfrak z}$ contained in the cube $[0,1)^3$, linked closely to the Zygmund dilation previously introduced,
satisfying the following properties.
\begin{itemize}
\item[(a)] $\sharp \mathcal P_{\mathfrak z} = 2^{4m+1}$ and $\, \displaystyle d_{\mathscr{D}_{\mathfrak{z}}} (p,q) \geq  \frac{1}{2^{2m}}$ 
for every pair of points $p, q \in \mathcal P_{\mathfrak z}$.
\item[(b)] $\sharp\mathcal Z_{\mathfrak z} \geq C m2^{4m} $.
\item[(c)] For each $z\in \mathcal Z_{\mathfrak z}$ there is exactly one point $p\in\mathcal P_{\mathfrak z}$ such that
$\, \displaystyle d_{\mathscr{D}_{\mathfrak{z}}} (p,z) = \frac{C}{2^{2m+k}}$.
\item[(d)] Let $R_\mathfrak{z}$ be a Zygmund dyadic rectangle and let 
$ R_0 = R_\mathfrak{z} \cap [0,1)^3 $. Then we have\footnote{Note that $R_0$ is a dyadic rectangle in $ \mathscr{D}^3$ 
and is not necessarily a Zygmund dyadic rectangle.}
\begin{itemize}
\item[$(i)$] if $\displaystyle |R_0| \geq \frac{1}{2^{4m + 2}} $, then 
$ \, \sharp( R_{0} \cap \mathcal P_{\mathfrak z}) = \sharp \mathcal P_{\mathfrak z} \cdot |R_{0}| $ and 
$ \, \sharp( R_{0} \cap\mathcal Z_{\mathfrak z}) \leq C k m 2^{4m} |R_{0}| $;
\item[$(ii)$] if $\displaystyle |R_0| < \frac{1}{2^{4m + 2}} $ and $R_0$ contains at least one point of 
$ \mathcal P_{\mathfrak z} $ and one point of $ \mathcal Z_{\mathfrak z} $, then 
$\, \sharp( R_{0} \cap \mathcal P_{\mathfrak z}) \leq C 2^k $, $\, \sharp( R_{0} \cap\mathcal Z_{\mathfrak z}) \leq C k 2^{k} $
and $\, \displaystyle |R_0| \geq \frac{C}{(2^{2m + k})^2} $.
\end{itemize}
\end{itemize}
\end{lem}
See Figure \ref{fig1} for a schematic diagram of $\mathcal P_{\mathfrak z}$.
 \begin{figure}[h]
 \centering
 \includegraphics[width=0.9\textwidth]{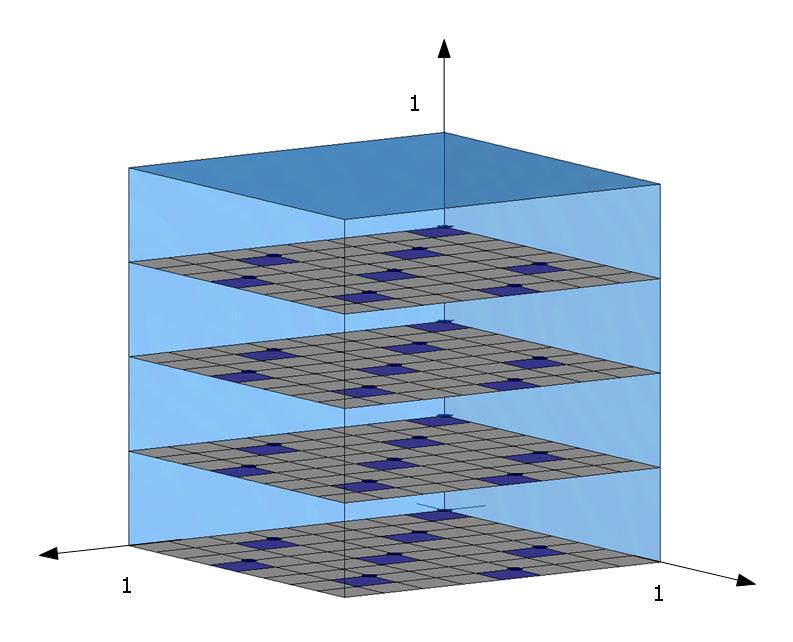}
 \caption{Schematic diagram indicating one of the examples of the locations of the points in $\mathcal P_{\mathfrak z}$, 
 which lie in identical copies of the intersection $\mathcal P$ with the $xy-$plane. These copies are placed at discrete 
 equally spaced heights, as if laid out on the floors of a multi-storey building.}\label{fig1}
 \end{figure}
 
In order to prove that there is no sparse 
domination in the tensor product setting $\mathbb R^n\times \mathbb R^m$ for the strong dyadic maximal function, the authors in \cite{BCOR} introduced two fundamental sets $\mathcal P$ and $\mathcal Z$ in $\mathbb R^2$ (see \cite[Theorem 2.1]{BCOR} for $\mathcal{P}$ and \cite[Theorem 2.2]{BCOR} for 
$\mathcal{Z}$).
 We have then made a construction adapted to the \emph{Zygmund dilatation} 
structure from these ones. 
In order to prove Lemma \ref{main lemma zygmund} and for the convenience of the reader, 
we state next the previously mentioned theorems.
\begin{thm}[Theorem 2.1, Theorem 2.2 and Remark 2.3 of \cite{BCOR}]
For every natural number $m$ and every natural number $k \ll m$ there exist two sets of points $\mathcal P$ and $\mathcal Z$
contained in $[0,1)^2$ satisfying the following properties.
\item[(a)] $\sharp \mathcal P = 2^{2m+1}$ and $\, \displaystyle d_{\mathscr{D}} (p,q) \geq \frac{1}{2^{m}}$ 
for every pair of points $p, q \in \mathcal P$.
\item[(b)] $\sharp\mathcal Z \geq C m2^{2m} $.
\item[(c)] For each $z\in \mathcal Z$ there is exactly one point $p \in \mathcal P$ such that
$\, \displaystyle \big( d_{\mathscr{D}} (p,z) \big)^2 = \frac{C}{2^{2m+k}}$.
\item[(d)] Let $R$ be a dyadic rectangle in $\mathscr{D}^2$. Then
\begin{itemize}
\item[$(i)$] if $\displaystyle |R| \geq \frac{1}{2^{2m+1}}$, we have $\, \sharp (R \cap \mathcal{P}) = \sharp P \cdot |R| $
and $\, \sharp (R \cap \mathcal{Z}) \leq C k m 2^{2m} |R|$;
\item[$(ii)$] if $\displaystyle |R| < \frac{1}{2^{2m+1}}$ and $R$ contains one point of $ \mathcal P $, we have 
$\, \sharp (R \cap \mathcal{Z}) \leq C k$.
\end{itemize}
The implied constants are independent of $k$ and $m$.
\end{thm}
See Figure \ref{fig2} and \ref{fig3} below for schematic diagrams of $\mathcal P$. See Figure \ref{figZ} next 
section for schematic diagrams of $\mathcal Z$.\\
\begin{figure}[h]
  \centering
  \begin{subfigure}[b]{0.4\linewidth}
    \includegraphics[width=\linewidth]{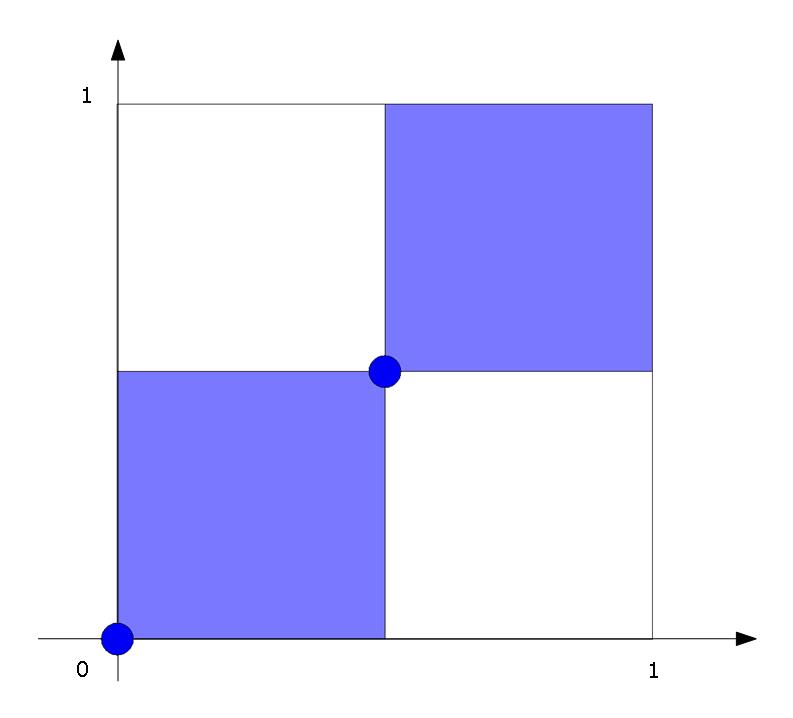} 
  \end{subfigure}  \qquad \qquad
  \begin{subfigure}[b]{0.4\linewidth}
    \includegraphics[width=\linewidth]{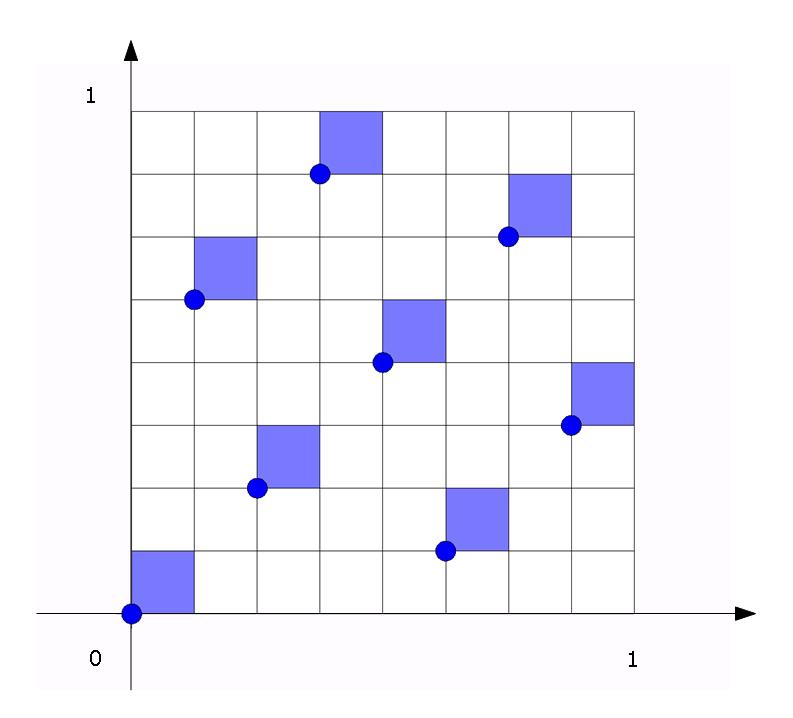}
  \end{subfigure}
 \caption{Schematic diagram indicating one of the examples of the locations of the points in $\mathcal P$ 
 for $m=0$ and $m=1$.}\label{fig2}
\end{figure} 
 \begin{figure}[h]
 \centering
 \includegraphics[width=0.6\textwidth]{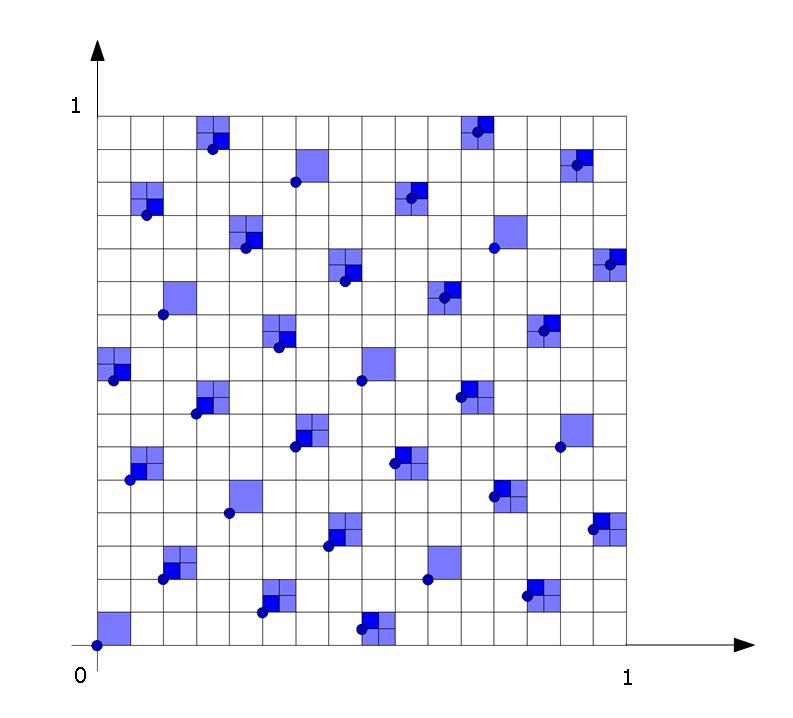}
 \caption{Schematic diagram indicating one of the examples of the locations of the points in $\mathcal P$
 for $m=2$.}\label{fig3}
 \end{figure}
\begin{proof}[Proof of Lemma \ref{main lemma zygmund}]
Let $ \mathcal P_{\mathfrak z} $ be the union of the level sets $ \mathcal P \times \{j \, 2^{-2m} \} $ for 
$j = 0, 1, \ldots, 2^{2m} -1$ and similarly for $\mathcal Z_{\mathfrak z}$. In particular, the $j$th level set $ \mathcal P \times \{j \, 2^{-2m} \} $ lies on the $j$th floor in Figure 1.
Thus, the items (a), (b) and (c) 
are an immediate consequence of the properties (a) for $\mathcal P$, (b)-(c) for $\mathcal Z$ and 
that the height of the Zygmund dyadic rectangles strictly contained in $[0,1)^3$ is low enough, see Figure \ref{figZygRec} below for indications of the Zygmund dyadic rectangles. 

 \begin{figure}[h]
 \centering
 \includegraphics[width=0.8\textwidth]{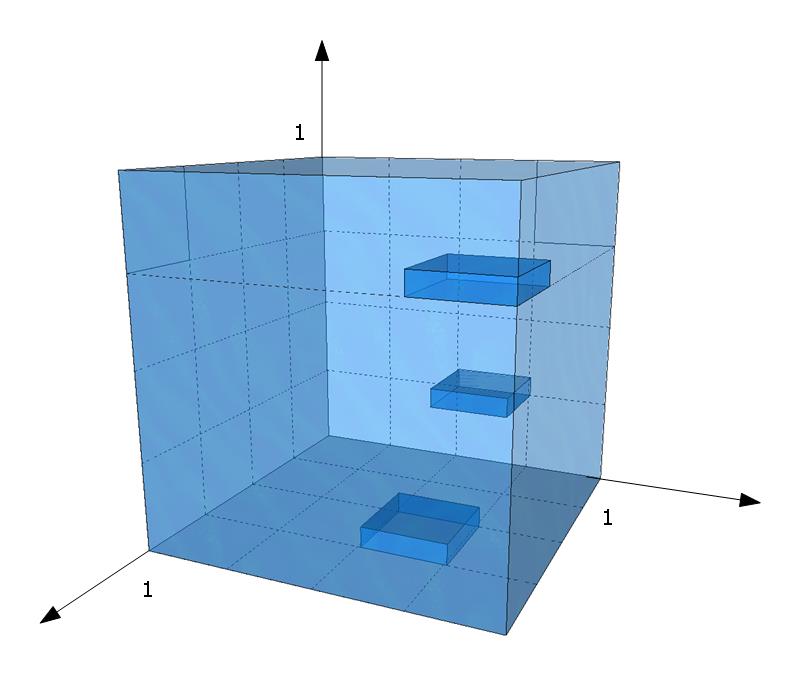}
 \caption{Schematic diagram indicating some examples of the Zygmund dyadic rectangles in each layer.}\label{figZygRec}
 \end{figure}

A point count and a pigeonholing argument allow us to obtain the item (d)-$(i)$, as noted in {Remark 2.3 of} \cite{BCOR}.

Now, let $R_0$ be as in (d)-$(ii)$ and set $R_0 = I \times J \times S$. Then $ |I \times J| \leq |S| $ because 
$R_{\mathfrak z}$ is a Zygmund dyadic rectangle and $S$ is the minimum between 1 and the height of $R_{\mathfrak z}$. 
Furthermore $ |I \times J| < 1 / 2^{2m+1} $ (otherwise we have a contradiction with $ |R_0| < 1 / 2^{4m + 2}$).
Also, set $p \in R_0 \cap \mathcal P_{\mathfrak z}$ and $z \in R_0 \cap \mathcal Z_{\mathfrak z}$. Without loss of generality
we can suppose that $p$ and $z$ have the same height, from the definitions of $ \mathcal P_{\mathfrak z} $ and 
$ \mathcal Z_{\mathfrak z} $. Then, $ |I \times J| = C / 2^{2m+k} $ by properties (a) for $\mathcal P$ and 
(c) for $\mathcal Z$. Therefore 
\begin{align*}
|R_0| \geq \frac{C}{(2^{2m + k})^2} \qquad \hbox{and} \qquad |S| < \frac{C 2^k}{2^{4m}}. 
\end{align*}
Finally, by property (d)-$(ii)$ for $ \mathcal Z $ we have that
\begin{align*}
 \sharp( R_{0} \cap \mathcal P_{\mathfrak z}) \leq C 2^k 
 \qquad \hbox{and} \qquad \sharp( R_{0} \cap\mathcal Z_{\mathfrak z}) \leq C k 2^{k} .
 \end{align*}
 The proof of Lemma \ref{main lemma zygmund} is complete.
\end{proof}
%%%%%%%%%%%%%%%%%%%%%%%%%%%%%%%%%%%%%%%%%%%%%%%%%%%%%%%%%%%%%%%%%%%%%%%%%%%%%
\begin{proof}[Proof of Theorem \ref{ZD}] Let $\mu$ and $\nu$ be the finite sums of point masses associated 
with $ \mathcal{P}_{\mathfrak z} $ and $ \mathcal{Z}_{\mathfrak z} $ of Lemma \ref{main lemma zygmund}, respectively.  Again, we have dropped the subscript $k$.

For each $z \in \mathcal Z_{\mathfrak z}$ there is only one point $p\in\mathcal P_{\mathfrak z}$ such that 
$\displaystyle d_{\mathscr{D}_{\mathfrak{z}}} (p,z) = C / 2^{2m+k}$, by Lemma \ref{main lemma zygmund} (c). So
there is a Zygmund dyadic rectangle $R_\mathfrak{z}$ containing both $p$ and $z$ such that 
$|R_\mathfrak{z}| = C / (2^{2m+k})^2$. Then, from the definition of $\mu$, ${\mathcal M}_{\frak z,d} \, \mu \, (z)$,  and 
by Lemma \ref{main lemma zygmund} (a), we have
\begin{align*}
{\mathcal M}_{\frak z,d} \, \mu (z) \geq \frac{\mu(R_\mathfrak{z})}{|R_\mathfrak{z}|} = { C (2^{2m+k})^2 \over 2^{4m+1}} = C \, 2^{2k} .
\end{align*}
Therefore,  from the definition of $\nu$,
\begin{align*}
\langle {\mathcal M}_{\frak z,d} \, \mu \, , \, \nu \rangle = 
\frac{1}{\sharp \mathcal{Z}_{\mathfrak z}} \sum_{z \in \mathcal Z_{\mathfrak z}} {\mathcal M}_{\frak z,d} \, \mu (z)
\geq  C \, 2^{2k} 
\end{align*}
and hence Theorem \ref{ZD} (a) is proved. 

We next show (b) of Theorem \ref{ZD}. Let  $R$ %$R_\mathfrak{z}$ 
be a Zygmund dyadic rectangle in $\mathcal S_{\mathfrak z}$ and 
let $ R_0 = R \cap [0,1)^3 $. % $ R_0 = R_\mathfrak{z} \cap [0,1)^3 $. 
Suppose now that $R_{0} = I \times J \times S$, so we shall consider the two cases 
corresponding to the items $(i)$ and $(ii)$ of Lemma \ref{main lemma zygmund} (d).

First, suppose that $\displaystyle |R_0| \geq \frac{1}{2^{4m + 2}} $. Then by (d)-$(i)$ and (a) of Lemma \ref{main lemma zygmund}, 
we have
\begin{align*}
\langle \mu \rangle_{R,r} = \frac{1}{\sharp \mathcal P_{\mathfrak z}} \bigg(  { \sharp ( R_0 \cap P_{\mathfrak z}) \over |R|} \bigg)^{1/r}
= \frac{1}{ 2^{(4m+1)(1 - 1/r)} } \bigg( { | R_0|\over |R| } \bigg)^{1/r} \leq \bigg( { | R_0|\over |R| } \bigg)^{1/r} .
\end{align*}
Next, by (d)-$(i)$ and (b) of Lemma \ref{main lemma zygmund}, we have
\begin{align*}
\langle \nu \rangle_{R,s} = \frac{1}{\sharp \mathcal Z_{\mathfrak z}} \bigg( { \sharp (R_0 \cap \mathcal Z_{\mathfrak z}) \over |R|} \bigg)^{1/s}
\leq \frac{C (km2^{4m})^{1/s} }{m2^{4m}} \bigg( { | R_0|\over |R| } \bigg)^{1/s} \leq C k^{1/s} \bigg( { | R_0|\over |R| } \bigg)^{1/s} .
\end{align*}
Thus, we get
\begin{align} \label{case-1}
\langle \mu \rangle_{R,r} \langle \nu \rangle_{R,s} \leq C k^{1/s} \bigg( { | R_0|\over |R| } \bigg)^{1/r  \,+ \, 1/s} .
\end{align}

Now, suppose that $\displaystyle |R_0| < \frac{1}{2^{4m + 2}} $ and 
$ \langle \mu \rangle_{R,r} \langle \nu \rangle_{R,s} > 0 $ (note that the cases 
$ \langle \mu \rangle_{R,r} \langle \nu \rangle_{R,s} = 0 $ contribute nothing to the sum on the
left-hand side of the inequality of Theorem \ref{ZD} (b)). So, $R_0$ contains at least one point of 
$ \mathcal P_{\mathfrak z} $ and one point of $ \mathcal Z_{\mathfrak z} $. Then by (d)-$(ii)$ and (a) of Lemma \ref{main lemma zygmund}, 
we have
\begin{align*}
\langle \mu \rangle_{R_0 , r} = \frac{1}{\sharp \mathcal P_{\mathfrak z}} \bigg(  { \sharp ( R_0 \cap P_{\mathfrak z}) \over |R_0|} \bigg)^{1/r}
\leq \frac{ C 2^{k/r} (2^{2m+k})^{2/r} }{2^{4m+1} } \leq C \, 2^{3k/r} .
\end{align*}
Also, by (d)-$(ii)$ and (b) of  Lemma \ref{main lemma zygmund}, we have
\begin{align*}
\langle \nu \rangle_{R_0 , s} = \frac{1}{\sharp \mathcal Z_{\mathfrak z}} \bigg( { \sharp (R_0 \cap \mathcal Z_{\mathfrak z}) \over |R_0|} \bigg)^{1/s} 
\leq \frac{C (k 2^k)^{1/s} (2^{2m+k})^{2/s} }{ m 2^{4m} } \leq \frac{C k^{1/s} \, 2^{3k/s} }{m} .
\end{align*}
As a consequence, we get that
\begin{align} \label{case-2}
\langle \mu \rangle_{R , r} \langle \nu \rangle_{R , s} = 
\langle \mu \rangle_{R_0 , r} \langle \nu \rangle_{R_0 , s} \bigg( {|R_0|\over |R|}\bigg)^{1/r \, + \, 1/s}
\leq \frac{C \, k^{1/s} \, 2^{3k(1/r \, + \, 1/s)}}{m} \bigg( {|R_0|\over |R|}\bigg)^{1/r \, + \, 1/s}.
\end{align}
Since $m = k 2^{6k}$, by \eqref{case-1} and \eqref{case-2} we have
\begin{align} \label{final}
\langle \mu \rangle_{R , r} \langle \nu \rangle_{R , s} \leq C \, k^{1/s} \left( 1 + \frac{1}{k} \right) \bigg( {|R_0|\over |R|}\bigg)^{1/r \, + \, 1/s}.
\end{align}   

We next split $\mathscr{S}_{\frak z}$ into the disjoint union of the subcollections
\begin{align*}
\mathscr{S}_{\frak z,j} = \big\{ R\in \mathscr{S}_{\frak z} : 2^{-j-1} |R| \leq |R_0| < 2^{-j} |R| \big\} 
\end{align*}
for $j = 0,1,\ldots$ We note that each of the rectangles $R$ in $\mathscr{S}_{\frak z,j}$
is contained in 
\begin{align*}
\Omega_j \vcentcolon = \bigg\{ z\in\mathbb R^3: \mathcal{M}_{\frak z,d} (\charac_{[0,1)^3}) (z)> {1\over 2^j} \bigg\}.
\end{align*}
The weak-type estimate of the Zygmund maximal function ($L\log L \to L^{1,\infty}$, see \cite{Co}) gives
$ |\Omega_j| \leq C \, j 2^j $. Finally, by \eqref{final} and $1/r + 1/s > 1$ we have
\begin{align*}
\sum_{R\in \mathscr{S}_{\frak z}} \langle \mu \rangle_{R,r} \langle \nu \rangle_{R,s} |R|
&=\sum_{j=0}^\infty\sum_{R\in \mathscr{S}_{\frak z,j}} \langle \mu \rangle_{R,r} \langle \nu \rangle_{R,s} |R|\\
&\leq C \, k^{1/s} \left( 1 + \frac{1}{k} \right) \sum_{j=0}^\infty \frac{1}{2^{j (1/r \, + \, 1/s)}} \sum_{R\in \mathscr{S}_{\frak z,j}} |R|\\
&\leq C \, k^{1/s} \left( 1 + \frac{1}{k} \right)  \sum_{j=0}^\infty \frac{|\Omega_j|}{\eta \, 2^{j (1/r \, + \, 1/s)}} \\
&\leq \frac{C}{\eta} \, k^{1/s} \left( 1 + \frac{1}{k} \right) \sum_{j=0}^\infty \frac{j 2^j}{ 2^{j (1/r \, + \, 1/s)}} \\
&\leq \frac{C}{\eta} \, k^{1/s} \left( 1 + \frac{1}{k} \right) ,
\end{align*}
as required.

The proof of Theorem \ref{ZD}  is complete.
\end{proof}

\section{Flag {dyadic} structure} \label{Flag}
Now we shall make some observations about the construction of the sets $\mathcal{P}$ and $\mathcal{Z}$
in \cite{BCOR}, which, together with the appropriate modifications regarding the exponents $r$ and $s$, will lead us 
to an immediate proof of Theorem \ref{Thm F}.

The construction of the set $\mathcal{P}$ is based on dyadic cubes in the plane and is compatible with the 
flag dyadic structure considered in this paper. So, we pick up the finite sum of point masses $\mu$ 
associated with this same $\mathcal{P}$. 

In order to get the set $\mathcal{Z}$ for fixed $k$ and $m$, the authors of \cite{BCOR} first consider dyadic rectangles $R$
of measure $2^{-2m-2}$. Then they choose special points of these rectangles to assemble $\mathcal{Z}$ 
(see Lemma 3.2 in \cite{BCOR}) and then prove that $ \sharp (R \cap \mathcal{Z}) \leq C k $ (see Lemma 3.3 in \cite{BCOR}). 
 \begin{figure}[h]
 \centering
 \includegraphics[width=0.5\textwidth]{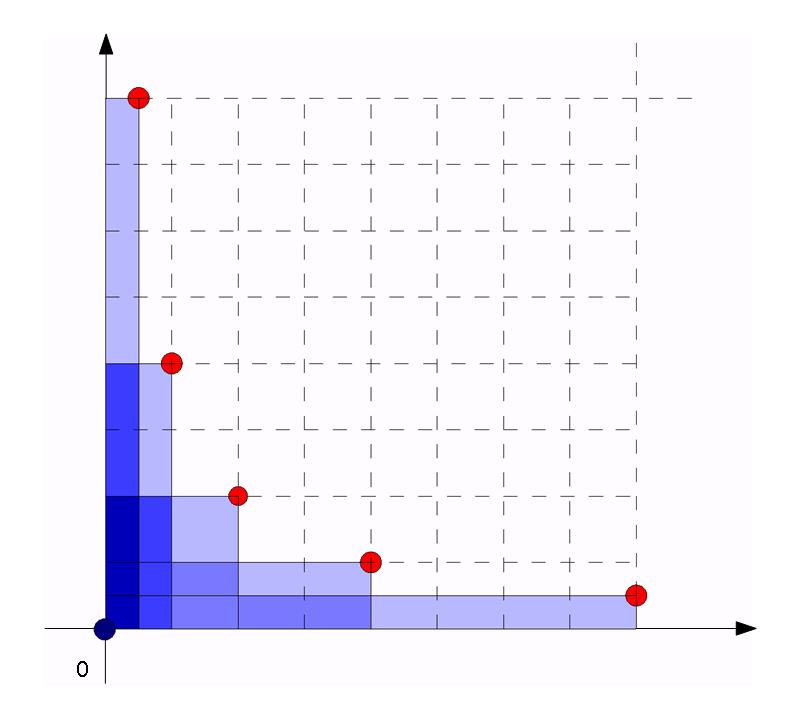}
 \caption{Schematic diagram indicating possible examples of the locations 
 of the points in $\mathcal Z$ fixed a point in $\mathcal P$.}\label{figZ}
 \end{figure}
Now, we can keep only those rectangles $R$ compatible with the flag dyadic structure considered here and
so we can build the set $\mathcal{Z}_{\rm flag}$ with the obvious modifications on the constants involved. 
After that, we take the finite sum of point masses $\nu$ associated with this new set $\mathcal{Z}_{\rm flag}$. 
Finally, the proof of Theorem \ref{Thm F} can be deduced following the steps previously carried out
for the dyadic maximal function $\mathcal{M}_{\mathfrak{z}, d}$ in Section \ref{Zygmund}.

\section{The strong {dyadic} maximal function and Sparse domination} \label{strong}
In this section, we  give a sketch of the proof of Theorem \ref{Thm ZS}. First, we modify
the proof of Proposition 2.6 in \cite{BCOR}, properly introducing the $L^r$-average and $L^s$-average
as we have done in the proof of Theorem \ref{ZD} (b). Then, following the procedure of 
the proof of Theorem \ref{Thm Z}, one can conclude Theorem \ref{Thm ZS} in the biparameter setting. 

Next, we modify the proof of Theorem 4.5 in \cite{BCOR}, properly introducing again the $L^r$-average and $L^s$-average. 
Then, using the previous step, one can conclude Theorem \ref{Thm ZS} for the full multiparameter setting.

\bigskip
{\bf Acknowledgement}
{The authors} would like to thank Jill Pipher, Guillermo Rey,  and Yumeng Ou for introducing us to the paper \cite{BCOR} and for many helpful discussions.

\end{document}